\documentclass[11pt, oneside]{article}   	
\usepackage{geometry}                		
\usepackage{graphicx}				
\usepackage{amssymb}
\usepackage{amsmath}
\usepackage{amsthm}
\usepackage{harpoon}
\usepackage{accents}
\usepackage{tikz}  
\usepackage{float}
\restylefloat{table}
\usepackage{mathrsfs}
\usepackage{verbatim}

\usetikzlibrary{positioning,chains,fit,shapes,calc}  

\newtheorem{theorem}{Theorem}
\newtheorem{lemma}{Lemma}
\newtheorem{definition}{Definition}

\newtheorem{corollary}{Corollary}

\title{On the Multiple Zeta Values $\zeta(\{ 2\}^k)$}
\author{Mario DeFranco}

\begin{document}
\maketitle

\abstract{We evaluate the multiple zeta values $\zeta(\{2\}^k)$ by proving a certain factorization property. The proof uses a combinatorial bijection and elementary telescoping series. We show how the infinite product for the sine function in fact implies its power series and other trigonometric properties. We define two constants, which we call pi-frequency and pi-amplitude, and show that they are equal  and satisfy the geometric definition of pi arising from the circumference of the circle.}

\section{Introduction} 

Multiple zeta values, also known as Euler sums, are certain infinite sums constructed from reciprocals of positive integers. They are important in various fields of mathematics, from number theory to  quantum physics. The general form of a multiple zeta value (``MZV" for short) is 
\[
\zeta(m_1, m_2, ..., m_k) = \sum_{1 \leq n_1 <n_2<...<n_k} \prod_{i=1}^k \frac{1}{n_i^{m_i}}
\]
where $m_i$ are positive integers and the convergence of the sum depends on the $m_i$. The term ``multiple zeta" comes from the fact that for $m>1$
\[
\zeta(m) = \sum_{n=1}^\infty \frac{1}{n^m}
\]
is a special value of the Riemann zeta function. We consider the MZV's of the form 
\[
\zeta(2,2,...,2) 
\]
where there are $k$ 2's, and denote them by $\zeta(\{2\}^k)$.  In this paper, we evaluate these $\zeta(\{2\}^k)$ using a combinatorial bijective proof. 

The evaluation of the zeta value $\zeta(\{2\}^1)=\zeta(2)$ in a closed form is known as the Basel problem. In A.D. 1735, L. Euler gave the solution 
\begin{equation} \label{basel}
\sum_{n=1}^\infty \frac{1}{n^2}= \frac{\pi^{2}}{6} 
\end{equation}
and the straightforward generalization 
\begin{equation} \label{mzv evaluation}
\zeta(\{2\}^k) = \frac{\pi^{2k}}{(2k+1)!}. 
\end{equation}
In the literature (see \cite{Gil}), the proof of this evaluation \eqref{mzv evaluation} traditionally follows that of Euler's which we briefly describe now.  
Euler used two expressions of the sine function and then compared the coefficients of $x$. First, he expressed $\sin(\pi x)$ as an infinite product
\begin{equation}\label{sin pi prod}
\frac{\sin(\pi x)}{\pi} =  x \prod_{n=1}^\infty (1-\frac{x^2}{n^2}). 
\end{equation}
Expanding out this product shows the coefficient of $x^{2k+1}$ to be 
\[
\zeta(\{2\}^k).
\]
Then he used the power series 
\begin{equation} \label{power series}
\frac{\sin(\pi x)}{\pi} = \sum_{k=0}^\infty (-1)^k \frac{\pi^{2k} x^{2k+1}}{(2k+1)!}. 
\end{equation}
Equating the coefficients yields \eqref{mzv evaluation}. We note that \eqref{mzv evaluation} is proved by other means by M. Hoffman in \cite{Hoffman} who attributes it as a conjecture to C. Moen. 

In this paper, we prove \eqref{mzv evaluation} using a combinatorial bijection and elementary telescoping sums. We start by considering the product
\begin{equation}\label{sin prod}
x\prod_{n=1}^\infty \frac{(n-x)(n+x)}{n^2}
\end{equation}  
and do not use its equivalence to sine. We do not assume the power series expansion \eqref{power series} of sine but in fact show that equation \eqref{power series} follows from the properties of the expression 
\eqref{sin prod}. From this context, we show how two constants, which we call pi-frequency and pi-amplitude, naturally arise, and then prove that these constants are equal to each other and to $\pi$, that is, the half-circumference of the unit circle. 
 
 We view these results as connecting the two key properties of the sine function: the ``periodic" property 
 \[
\sin(x+ \pi) = -\sin(x)
 \]
 and the ``derivative" property
 \[
 \frac{d^2}{dx^2} \sin(x) = -\sin(x).
 \]
 The interplay between these two properties makes the sine function central in many fields, such as complex analysis. \textit{A priori} it is not obvious that the infinite product \eqref{sin prod} should have the power series \eqref{power series}, or that the power series \eqref{power series} should be periodic with zeros at the integers. This paper thus helps to understand this connection on a combinatorial level. 
 
\section{$F(x)$ and $F''(x)$}
We start by considering the following infinite product which is a natural construction of a 2-periodic function whose zero set is the integers.
\begin{definition}
Define the function $F(x)$ by 
\begin{align*}
F(x) &= x \prod_{n=1}^\infty (\frac{n-x}{n})(\frac{n+x}{n})\\ 
&= x \prod_{n=1}^\infty (1-\frac{x^2}{n^2}).
\end{align*}
\end{definition}
\begin{lemma} 
$F(x)$ is an entire function whose zero set is $\mathbb{Z}$ and which satisfies 
\[
F(x+1) = -F(x).
\]
\end{lemma}
\begin{proof}
Let 
\[
F_N(x) = x \prod_{n=1}^N (1- \frac{x^2}{n^2}).
\]
Since 
\[
\sum_{n=1}^{\infty} |\frac{x^2}{n^2}|
\]
is absolutely convergent for any $x \in \mathbb{C}$, it follows from a standard result in complex analysis that $F_N(x) \rightarrow F(x)$ uniformly on compact sets $K$, and that the infinite product $F(x)$ defines an entire function which is zero if and only if one of its factors is zero. 

Finally 
\[
F_N(x+1) = \frac{N+1+x}{N-x} F_N(x).
\] 
Taking the limit as $N \rightarrow \infty$ gives 
\[
F(x+1)=-F(x).
\]
This completest the proof.
\end{proof}
\begin{lemma}
\[
F(x) = \sum_{k=0}^\infty (-1)^k \zeta(\{ 2\}^k)x^{2k+1}
\] 
where $\zeta(\{ 2\}^0)$ denotes 1.
\end{lemma}
\begin{proof}
This follows from expanding out the $F_N(x)$ and letting $N \rightarrow \infty$. This is justified by the uniform convergence the $F_N(x)$. 
\end{proof}
\begin{lemma} \label{F p} 
\[
F''(x) = -F(x)  p(x)
\]
where
\[
p(x) = 6(\sum_{n=1}^\infty \frac{1}{n^2(1-\frac{x^2}{n^2})}) - 8 x^2 \sum_{l_2=l_1+1}^\infty \sum_{l_1=1}^\infty \frac{1}{l_1^2 l_2^2(1-\frac{x^2}{l_1^2})(1-\frac{x^2}{l_2^2})}.
\]
\end{lemma}
\begin{proof}
We successively apply two differentiations $\displaystyle \frac{d}{dx}$ to the product 
\[
F_N(x) =  x \prod_{n=1}^N (1-\frac{x^2}{n^2}).
\] 
By the product rule we have 
\[
F_N'(x) = F_N(x)(\frac{1}{x}-2x \sum_{n=1}^N \frac{1}{n^2(1-\frac{x^2}{n^2})})
\]
and
\begin{equation} \label{f1 f2}
F_N''(x) = F_N(x) \big( (\sum_{f_1\neq f_2} \frac{f_1'(x) f_2'(x)}{f_1(x)f_2(x)})+\sum_{f_1}\frac{f_1''(x)}{f_1(x)}\big)
\end{equation}
where $f_1(x)$ and $f_2(x)$ range over the set of functions consisting of
\[
x \text{ and } (1-\frac{x^2}{n^2})
\]
for $1\leq n \leq N$. For fixed $n$, in the first sum in \eqref{f1 f2}, the pairs $(f_1, f_2)$  of the form 
\[
(x, 1-\frac{x^2}{n^2}) \text{ and } ( 1-\frac{x^2}{n^2}, x)
\]
contribute in total
\[
-\frac{4}{n^2(1-\frac{x^2}{n^2})}.
\]
For fixed $n$, in the second sum in \eqref{f1 f2}, the choice 
\[
f_1(x) = (1-\frac{x^2}{n^2})
\]
contributes 
\[
-\frac{2}{n^2(1-\frac{x^2}{n^2})}
\]
and the choice 
\[
f_1(x) = x
\]
contributes 0. 
For fixed $l_1\neq l_2$, in the first sum in \eqref{f1 f2}, the pairs $(f_1, f_2)$  of the form 
\[
(1-\frac{x^2}{l_1^2}, 1-\frac{x^2}{l_2^2}) \text{ and } ( 1-\frac{x^2}{l_2^2}, 1-\frac{x^2}{l_1^2})
\]
contribute in total
\[
\frac{8x^2}{l_1^2 l_2^2(1-\frac{x^2}{l_1^2})(1-\frac{x^2}{l_2^2})}.
\]
Now it is another standard result of complex analysis (see \cite{Lang}) that, for a sequence of holomorphic $f_N(x)$, the uniformly convergence of $f_N(x) \rightarrow f(x)$ on a compact set $K$ implies the uniform convergence of $f_N'(x) \rightarrow f'(x)$ on $K$. This completes the proof.

\end{proof}
The next theorem implies the recursive factorization of $\zeta(\{ 2\}^k)$ after comparing coefficients of a power series. A key step is using a certain elementary telescoping series. We turn the proof into a combinatorial bijective proof in Section \ref{bijection}.

\begin{theorem} \label{p constant}
Let $p(x)$ denote the function
\begin{equation} \label{p}
p(x) = 6(\sum_{n=1}^\infty \frac{1}{n^2(1-\frac{x^2}{n^2})}) - 8 x^2 \sum_{l_2=l_1+1}^\infty \sum_{l_1=1}^\infty \frac{1}{l_1^2 l_2^2(1-\frac{x^2}{l_1^2})(1-\frac{x^2}{l_2^2})}. 
\end{equation}
Then $p(x)$ is constant in $x$, and 
\[
p(x) = p(0) = 6\sum_{n=1}^\infty \frac{1}{n^2}.
\]
\end{theorem} 
\begin{proof}
We prove that the coefficient of $x^{2j}$ in $p(x)$ is 0 for $j \geq 1$. It is sufficient to assume $0 <x <1$. We will use the bounds 
\[
\sum_{n=1}^\infty \frac{1}{n^{2j}} < 2
\]
for $j \geq 1$. These bounds follow from comparing the sum with its integral. 

We first consider the second sum in expression $\eqref{p}$. 
By partial fraction decomposition, 
\begin{align*}
 &\frac{8 x^2}{l_1^2 l_2^2(1-\frac{x^2}{l_1^2})(1-\frac{x^2}{l_2^2})} = 4(\frac{1}{l_1-l_2}-\frac{1}{l_1+l_2})\frac{1}{l_2(1-\frac{x^2}{l_2^2})}+4(\frac{1}{l_2-l_1}-\frac{1}{l_2+l_1})\frac{1}{l_1(1-\frac{x^2}{l_1^2})}\\ 
 &= \sum_{j=1}^\infty  4(\frac{1}{l_1-l_2}-\frac{1}{l_1+l_2})\frac{1}{l_2} \frac{x^{2j}}{l_2^{2j}} +4(\frac{1}{l_2-l_1}-\frac{1}{l_2+l_1})\frac{1}{l_1}\frac{x^{2j}}{l_1^{2j}}
\end{align*}

Thus 
\begin{equation} \label{8 l1 l2}
\sum_{l_2=l_1+1}^\infty \sum_{l_1=1}^\infty \frac{8 x^2}{l_1^2 l_2^2(1-\frac{x^2}{l_1^2})(1-\frac{x^2}{l_2^2})}  = \sum_{l_2=l_1+1}^\infty \sum_{l_1=1}^\infty \sum_{j=1}^\infty  4(\frac{1}{l_1-l_2}-\frac{1}{l_1+l_2})\frac{1}{l_2} \frac{x^{2j}}{l_2^{2j}} +4(\frac{1}{l_2-l_1}-\frac{1}{l_2+l_1})\frac{1}{l_1}\frac{x^{2j}}{l_1^{2j}}
\end{equation}
We claim that we may interchange the order of summation. We have  
\begin{align*}
 &\sum_{l_2=l_1+1}^\infty \sum_{l_1=1}^\infty \sum_{j=1}^\infty  \bigl |4(\frac{1}{l_1-l_2}-\frac{1}{l_1+l_2})\frac{1}{l_2} \frac{1}{l_2^{2j}}\bigr | x^{2j}\\ 
 =& 8\sum_{j=1}^\infty  x^{2j} \sum_{l_1= 1}^\infty \sum_{k=1}^\infty  (\frac{1}{k(2l_1+k)}) \frac{1}{(l_1+k)^{2j}} \\ 
 \leq& 8 \sum_{j=1}^\infty  x^{2j} (\sum_{k=1}^\infty  \frac{1}{k^2}) (\sum_{l_1 = 1}^\infty\frac{1}{l_1^{2j}} ) \\ 
\leq& \frac{8\cdot2\cdot2x^2}{1-x^2}
\end{align*}

We have 
\begin{align*}
 &\sum_{l_2=l_1+1}^\infty \sum_{l_1=1}^\infty \sum_{j=1}^\infty  \bigl |4(\frac{1}{l_2-l_1}-\frac{1}{l_2+l_1})\frac{1}{l_1} \frac{1}{l_1^{2j}}\bigr | x^{2j}\\ 
 =&8 \sum_{j=1}^\infty x^{2j}\sum_{l_1=1}^\infty\sum_{k= 1}^\infty  \frac{1}{k(2l_1+k)} \frac{1}{l_1^{2j}} \\ 
 \leq & 8 \sum_{j=1}^\infty x^{2j}(\sum_{k=1}^\infty  \frac{1}{k^2}) (\sum_{l_1= 1}^\infty\frac{1}{l_1^{2j}} )\\ 
\leq& \frac{8\cdot2\cdot2x^2}{1-x^2}
 \end{align*}
Thus by absolute convergence we may re-arrange the sum \eqref{8 l1 l2} to obtain 
\begin{align*}
& \sum_{j=1}^\infty x^{2j}  \left( \bigl( \sum_{l_2=l_1+1}^\infty \sum_{l_1=1}^\infty  4(\frac{1}{l_1-l_2}-\frac{1}{l_1+l_2})\frac{1}{l_2} \frac{1}{l_2^{2j}}\bigr) +  \bigl( \sum_{l_2=l_1+1}^\infty \sum_{l_1=1}^\infty 4(\frac{1}{l_2-l_1}-\frac{1}{l_2+l_1})\frac{1}{l_1}\frac{1}{l_1^{2j}} \bigr) \right).
\end{align*}
 
Now let $H(n)$ denote the harmonic number
\[
H(n) =\sum_{k=1}^n \frac{1}{k}. 
\]
Thus 
\begin{align} \nonumber
&-8x^2\sum_{l_2=l_1+1}^\infty \sum_{l_1=1}^\infty \frac{1}{l_1^2 l_2^2(1-\frac{x^2}{l_1^2})(1-\frac{x^2}{l_2^2})} \\ \label{nl2}
=&\sum_{n=1}^\infty \sum_{j=1}^\infty  (-4(\frac{x^2}{n^2})^j\sum_{ l_2>n}(\frac{1}{l_2-n}-\frac{1}{l_2+n})\frac{1}{n} \\ \label{nl1} 
 &+\sum_{j=1}^\infty  -4(\frac{x^2}{n^2})^j \sum_{n> l_1\geq 1}(\frac{1}{l_1-n}-\frac{1}{l_1+n})\frac{1}{n}.
\end{align}

At expression \eqref{nl2}, for each $n$, the coefficient of $x^{2j}$ is a telescoping series that evaluates to  
\begin{equation}\label{H l2}
-4 H(2n)\frac{1}{n^{2j+1}}. 
\end{equation}
At expression \eqref{nl1}, for each $n$, the coefficient of $x^{2j}$ is a finite sum that evaluates to
\begin{equation} \label{H l1}
4(H(2n-1) - \frac{1}{n}) \frac{1}{n^{2j+1}}. 
\end{equation}
Adding expressions \eqref{H l2} and \eqref{H l1} gives 
\[
-6\frac{1}{n^{2j+2}}.
\]
The coefficient of $x^{2j}$ for $j \geq 1$ in  
\[
6 \sum_{n=1}^\infty \frac{1}{n^2(1-\frac{x^2}{n^2})}
\]
is 
\[
6\frac{1}{n^{2j+2}}.
\]
This completes the proof. 
\end{proof}
\begin{corollary} \label{corollary}
\[
(2k+1)(2k)\zeta(\{ 2\}^k) = \zeta(\{ 2\}^{k-1}) (6 \sum_{n=1}^\infty \frac{1}{n^2}).
\]
and therefore
\[
\zeta(\{ 2\}^k) =\frac{(6 \sum_{n=1}^\infty \frac{1}{n^2})^k }{(2k+1)!}.
\]

\end{corollary} 
\begin{proof}
From 
\[
F(x) = \sum_{k=0}^\infty (-1)^k \zeta(\{ 2\}^k)x^{2k+1},
\]
Lemma \ref{F p} and Theorem \ref{p constant} imply that
\[
F''(x) = F(x) (6 \sum_{n=1}^\infty \frac{1}{n^2}).
\]
Comparing coefficients proves the corollary.
\end{proof}

We make the proof of the above corollary combinatorial in the next section.

\section{Bijection for $\zeta(\{2\}^k)$} \label{bijection}
We fix integer $k \geq 2$. We show how to add and subtract terms to $(2k+1)(2k)\zeta(\{ 2\}^k)$ to factor out $6\zeta(2)$. (The terms that we add and subtract are referred to as $\alpha$-components below.) We index the terms in $(2k+1)(2k)\zeta(\{ 2\}^k)$ by defining a graph $G(k)$ with infinite vertex set $V(k)$.  
\subsection{The vertices $V(k)$}

Let $k\geq2$. Define $V_1(k)$ to be the set of vertices $v$
 \[
 v = (\mu; n)
 \]
where $n$ is any positive integer, and $\mu$ is a set of positive integers 
\[
\mu = \{\mu(1), \mu(2), ..., \mu(j) \}
\]
with $\mu(i) < \mu(i+1)$ and  $0 \leq j = |\mu| \leq k-1$. We say that $v$ is 1-distinct if $n \notin \mu$ and 0-distinct if $n \in \mu$.  We say that $v$ is of order $j$ and write $|v|=j$. 

Define $V_2(k)$ to be the set of vertices $v$ 
\[
v = (\mu; l_1, l_2; \epsilon)
\]
where $l_1<l_2$ are any positive integers; $\epsilon \in \{ 1,2\}$; and $\mu$ is a set of positive integers  
\[
\mu = \{\mu(1), \mu(2), ..., \mu(j) \}
 \]
 with $\mu(i) < \mu(i+1)$  and $0 \leq j = |\mu| \leq k-2$. We say that $v$ is 2-distinct if $l_1 \notin \mu$ and $l_2 \notin \mu$. We say that $v$ is 1-distinct if exactly one of $l_1$ and $l_2$ is in $\mu$. We say that $v$ is 0-distinct if both $l_1$ and $l_2$ are in $\mu$. We say that $v$ is of order $j$ and write $|v|=j$. 

Define 
\[
V(k) = V_1(k) \cup V_2(k).
\]

Define the function 
\[
t_k: V(k) \rightarrow \mathbb{R}
\]
 by
\[
t_k(\{\mu(1), \mu(2), ..., \mu(j) \}; n) = 6(-1)^{j+1} (\prod_{i=1}^{j} \frac{1}{\mu(i)^2}) (\frac{1}{n^2})^{k-j}
\]
and 
\begin{align*}
t_k(\{\mu(1), \mu(2), ..., \mu(j) \}; l_1, l_2; \epsilon) &= 4(-1)^j  (\prod_{i=1}^j \frac{1}{\mu(i)^2}) \frac{1}{l_\epsilon^{2(k-j)-1}}(\frac{(-1)^{\epsilon-1}}{l_2 - l_1} - \frac{1}{l_1 + l_2})\\ 
                                                                                    &= 8(-1)^j  (\prod_{i=1}^j \frac{1}{\mu(i)^2})\frac{(-1)^{\epsilon-1}}{l_\epsilon^{2(k-j-1)}(l_2^2 - l_1^2)}. 
\end{align*}

  

\begin{lemma} 
For $k \geq 2$, the sum
\[
\sum_{v \in V(k)} t_k(v)
\]
is absolutely convergent. 
\end{lemma} 
\begin{proof}

Summing over all vertices $V_1(k)$, both $0$-distinct and $1$-distinct, we have 
\[
\sum_{v \in V_1(k),\,\, |v| = j} |t_k(v)|  \leq 6\zeta(\{ 2\}^j) \zeta(2(k-j)) 
\]
Likewise summing over all vertices in $V_2(k)$ for fixed $l_1<l_2$ gives
\[
\sum_{|\mu|=j}|t(\mu; l_1, l_2, 1)| = 8\zeta(\{2\}^j) \frac{1}{l_1^{2(k-j-1)}(l_2^2 - l_1^2)}
\]
and
\[
\sum_{|\mu|=j}|t(\mu; l_1, l_2, 2)| = 8\zeta(\{2\}^j) \frac{1}{l_2^{2(k-j-1)}(l_2^2 - l_1^2)}.
\]
Thus
\begin{align*}
\sum_{v \in V_2(k), \, \, |v| = j} |t_k(v)|  &= \sum_{l_2 = l_1+1}\sum_{l_1=1} \big( 8\zeta(\{2\}^j) \frac{1}{l_1^{2(k-j-1)}(l_2^2 - l_1^2)}+  8\zeta(\{2\}^j) \frac{1}{l_2^{2(k-j-1)}(l_2^2 - l_1^2)}\big)\\
& \leq 16\zeta(\{2\}^j) \sum_{n=1}^\infty \sum_{l_1=1}^\infty\frac{1}{l_1^{2(k-j-1)}((l_1+n)^2-l_1^2)} \\ 
& \leq 16\zeta(\{2\}^j) \sum_{n=1}^\infty \sum_{l_1=1}^\infty\frac{1}{l_1^{2(k-j-1)}(2nl_1+n^2)} \\ 
& \leq 16\zeta(\{2\}^j) (\sum_{n=1}^\infty \frac{1}{n^2})(\sum_{l_1=1}^\infty\frac{1}{l_1^{2(k-j-1)}}) \\ 
& =  16\zeta(\{2\}^j)\zeta(2)\zeta(2(k-j-1))
\end{align*}
This completes the proof. 
\end{proof}

We next draw edges on $V(k)$. The edges are of two types, $\alpha$ and $\beta$. 

\subsection{The $\alpha$-edges} 

\begin{definition} 
Let $\mu$ be a set of positive integers 
\[
\mu = \{ \mu(1), \mu(2), ..., \mu(j)\}
\]
with $\mu(i) < \mu(i+1)$. Let $\hat{\mu}_i$ denote the set
\[
\mu = \{ \mu(1),...,\mu(i-1), \mu(i+1), ...,  \mu(j)\}.
\]
\end{definition}

All sets have order less than $k$. 

1. For each set $\mu$, draw an $\alpha$-edge between the vertices 
\[
(\mu; \mu(i)) \text{ and }(\hat{\mu}_i; \mu(i)). 
\]

2. For each set $\mu$, draw an $\alpha$-edge between the vertices 
\[
(\mu; l_1,l_2; 1) \text{ and } (\mu; l_1,l_2; 2). 
\]

3. Suppose we have sets $\lambda, \mu$ and $\nu$ such that 
\[
\lambda = \hat{\mu}_i=\hat{\nu}_j
\]
for some (possibly equal) indices $i$ and $j$. Also suppose that $\mu(i) \notin \nu$ and $\nu(j) \notin \mu$ and, without loss of generality, that $\mu(i)<\nu(j)$. Then draw three $\alpha$-edges between the three vertices 
\[
(\mu; \mu(i),\nu(j); 1), \,\,\,(\nu; \mu(i) , \nu(j); 1), \text{ and } (\lambda; \mu(i), \nu(j) ;1).
\]

4. For each set $\rho$ and indices $i<j$, draw an $\alpha$-edge between the vertices 
\[
(\rho; \rho(i), \rho(j); 1) \text{ and } (\hat{\rho}_j; \rho(i), \rho(j); 1). 
\]

Let $G(k, \alpha)$ be the graph with vertex set $V(k)$ with the $\alpha$-edges. Let an $\alpha$-component refer to a component of $G(k, \alpha)$ with at least one $\alpha$-edge.  

\subsection{The $\beta$-edges}

1. Suppose 
\[
u = (\mu; n)
\]
is a 1-distinct vertex with $2\leq |\mu|\leq k-2$. Then draw $\beta$-edges between $u$ and all 2-distinct vertices of the form 
\[
(\mu; n,l_2; 1) \text{ and } (\mu; l_1,n; 2).
\]
Draw $\beta$-edges between $u$ and all 1-distinct vertices of the form 
\[
(\mu; n,\mu(i); 1) \text{ and } (\mu; \mu(i'),n; 2).
\]
for any indices $i$, $i'$ of $\mu$ (so $n \notin \mu$). 

2. Suppose 
\[
u = (\mu; n)
\]
is a 0-distinct vertex with $2\leq |\mu|\leq k-2$. Then draw $\beta$-edges between $u$ and all 1-distinct vertices of the form 
\[
(\mu; n,l_2; 1) \text{ and } (\mu; l_1,n; 2)
\]
(so $n \in \mu$ and  $l \notin \mu$ ). Draw $\beta$-edges between $u$ and all 0-distinct vertices of the form 
\[
(\mu; n,\mu(i); 1) \text{ and } (\mu;\mu(i'),n; 2).
\]
for any indices $i$, $i'$ of $\mu$ (now $n \in \mu$). 

3. Suppose 
\[
u = (\mu; n)
\]
is a 1-distinct vertex with $|\mu|=1$. Then draw $\beta$-edges between $u$ and all 2-distinct vertices of the form 
\[
(\mu; n,l_2; 1) \text{ and } (\mu; l_1,n; 2). 
\]
Draw a $\beta$-edge between $u$ and the 1-distinct vertex of the form
\[
(\mu; n,\mu(1); 1), \text{ if } n < \mu(1),
\]
or draw a $\beta$-edge between $u$ and the 1-distinct vertex of the form
\[
(\mu; \mu(1),n; 2), \text{ if } n > \mu(1).
\]
4. Suppose 
\[
u = (\mu; \mu(1))
\]
is a 0-distinct vertex with $|\mu|=1$. Then draw $\beta$-edges between $u$ and all 1-distinct vertices of the form 
\[
(\mu; \mu(1), l_2; 1) \text{ and } (\mu; l_1,\mu(1); 2). 
\]

5. Suppose 
\[
u = (\emptyset; n).
\]
The draw $\beta$-edges between $u$ and all vertices of the form 
\[
(\emptyset; l_1, l_2; 1) \text{ and } (\emptyset; l_1, l_2; 2).
\]

Let $G(k, \beta)$ be the graph with vertex set $V(k)$ with the $\beta$-edges. Let a $\beta$-component refer to a component of $G(k, \beta)$ with at least one $\beta$-edge.  

\begin{theorem}
\[
(2k+1)(2k)\zeta(\{ 2\}^k) = \zeta(\{ 2\}^{k-1}) (6 \sum_{n=1}^\infty \frac{1}{n^2} )
\]
\end{theorem}
\begin{proof} 

Let $C_\alpha$ be an $\alpha$-component of $G(k,\alpha)$. Then there are only finitely many vertices in $C_\alpha$ and it is straightforward to check that
\[
\sum_{v \in C_\alpha} t(v)=0.
\]
Let $C_\beta$ be an $\beta$-component of $G(k,\beta)$. Then there are infinitely vertices $C_\beta$, and, using the telescoping series in the proof of Theorem \ref{p constant}, we check that
\[
\sum_{v \in C_\beta} t(v)=0.
\]

The only vertices not in an $\alpha$-component are 1-distinct vertices of the form 
\[
(\mu; n)
\] 
with $|\mu|=k-1$ and 2-distinct vertices of the form 
\[
(\mu; l_1,l_2; 1) \text{ and } (\mu; l_1,l_2; 2) 
\] 
with $|\mu|=k-2$.  This implies that 
\begin{equation}\label{t sum alpha}
(-1)^k\sum_{v \in V(k)} t(v)=  (6\sum_{1\leq n_1 <...<n_{k-1}; 1\leq n_k } \prod_{i=1}^k \frac{1}{n_i^2})  + (8\sum_{1\leq n_1 <...<n_{k-2}; 1\leq n_{k-1}<n_ k} \prod_{i=1}^k \frac{1}{n_i^2})
\end{equation}
The identity 
\[
6 k+ 8{ k \choose 2} = (2k+1)(2k)
\]
implies that the expression \eqref{t sum alpha} is equal to 
\[
(2k+1)(2k) \zeta(\{ 2\}^k).
\]

The only vertices not in a $\beta$-component are 1-distinct and 0-distinct vertices of the form 
\[
(\mu; n)
\] 
with $|\mu|=k-1$. 
This implies that 
\[
(-1)^k\sum_{v \in V(k)} t(v) = \zeta(\{ 2\}^{k-1})6\sum_{n=1}^\infty \frac{1}{n^2}. 
\]
This completes the proof.
\end{proof}

\section{Connection to $\pi$}
We now show that $F(x)$ can be used to parametrize the unit circle. We define constants pi-frequency $\pi_{\mathrm{freq}}$ and pi-amplitude $\pi_{\mathrm{amp}}$. These names derive from the function
\[
G(x) = \pi_{\mathrm{amp}} F(\frac{x}{\pi_{\mathrm{freq}}})
\]
which we prove to be equal to $\sin(x)$, that is, the $y$-coordinate of a point on the unit circle. 

\begin{definition} 
Define the number pi-frequency $\pi_{\mathrm{freq}}$ by 
\[
\pi_{\mathrm{freq}} = (6 \sum_{n=1}^\infty \frac{1}{n^2})^{\frac{1}{2}}.
\]
\end{definition}

Thun we express by Corollary \ref{corollary}
\begin{equation} \label{F power series}
F(x) = \sum_{k=0}^\infty (-1)^k \pi_{\mathrm{freq}}^{2k}\frac{x^{2k+1}}{(2k+1)!}
\end{equation}
and
\[
F''(x) = - \pi_{\mathrm{freq}}^2 F(x).
\]
\begin{lemma}\label{F inc}
The maximum of the function $ F(x)$ occurs at $x = \frac{1}{2}$. The function $F(x)$ is increasing on $[0,\frac{1}{2}]$ and decreasing on $[\frac{1}{2},1]$. 
\end{lemma}
\begin{proof}
We may express $F(x)$ as
\[
F(x) = \prod_{n=0}^\infty\frac{(n+x)(n+1-x)}{(n+1)^2}. 
\]
Each factor 
\[
(n+x)(n+1-x)
\]
is increasing on $[0,\frac{1}{2}]$ and decreasing on $[\frac{1}{2},1]$. This completes the proof.

\end{proof}

\begin{definition} 
Define the number pi-amplitude $\pi_{\mathrm{amp}}$ by 
\[
\pi_{\mathrm{amp}} = F(\frac{1}{2})^{-1} =  2 \prod_{n=1}^\infty (\frac{2n}{2n-1})(\frac{2n}{2n+1})
\]
\end{definition}

We note that since $x=\frac{1}{2}$ is not a zero of $F(x)$, we have that $\pi_{\mathrm{amp}}$ is finite and non-zero.
\begin{theorem} 
Let $\pi$ denote the arc-length of the half-circle of radius 1. Then
\[
\pi_{\mathrm{amp}} = \pi_{\mathrm{freq}} = \pi
\]
\end{theorem}
\begin{proof} 
We first prove 
\[
\pi_{\mathrm{amp}} = \pi_{\mathrm{freq}}.
\]
Let $G(x)$ denote the function 
\[
G(x) = \pi_{\mathrm{amp}} F(\frac{x}{\pi_{\mathrm{freq}}}).
\]
By the definition of $\pi_{\mathrm{amp}}$, the function $G(x)$ has maximum 1 for all $x \in \mathbb{R}$ and is attained at $\displaystyle x = \frac{\pi_{\mathrm{freq}}}{2}$. From equation \eqref{F power series}, we have
\[
G(x) = \frac{\pi_{\mathrm{amp}}}{\pi_{\mathrm{freq}}}\sum_{k=0}^\infty (-1)^k \frac{x^{2k+1}}{(2k+1)!}. 
\] 
Let $g(x)$ denote the function
\[
g(x) = \sum_{k=0}^\infty (-1)^k \frac{x^{2k+1}}{(2k+1)!}.
\]
Thus 
\[
g''(x) = -g(x).
\]
Then 
\begin{equation} \label{g g' 1}
g(x)^2 + g'(x)^2 = 1.
\end{equation}
This follows because the derivative of the left side is 0, and $g(0)=0$ and $g'(0)=1$. 

It follows that 
\[
G(x)^2 + G'(x)^2 = (\frac{\pi_{\mathrm{amp}}}{\pi_{\mathrm{freq}}})^2.
\]
Evaluating at $\displaystyle x = \frac{\pi_{\mathrm{freq}}}{2}$, a local maximum of $G(x)$, gives 
\[
G(\frac{\pi_{\mathrm{freq}}}{2}) = 1 \text{ and } G'( \frac{\pi_{\mathrm{freq}}}{2}) = 0, 
\]
which shows 
\[
 (\frac{\pi_{\mathrm{amp}}}{\pi_{\mathrm{freq}}})^2 = 1.
\]
This proves 
\[
\pi_{\mathrm{amp}} = \pi_{\mathrm{freq}}.
\]

Now $G(x)$ is increasing on $[-\frac{\pi_{\mathrm{freq}}}{2}, \frac{\pi_{\mathrm{freq}}}{2}]$ and $G'(x)$ is even and increasing on $[-\frac{\pi_{\mathrm{freq}}}{2}, 0]$. This follows from the periodicity of $F(x)$ and Lemma \ref{F inc}. Since 
\[
G(x)^2 + G'(x)^2=1,
\]
 we have that 
\[
(G(x), G'(x)), \, \, \, -\frac{\pi_{\mathrm{freq}}}{2}\leq x\leq \frac{\pi_{\mathrm{freq}}}{2}
\]
is a parametrization of the half-circle of radius 1 in the upper-half plane. Computing its arc-length gives 
\begin{align*}
\int_{-\frac{\pi_{\mathrm{freq}}}{2}}^{\frac{\pi_{\mathrm{freq}}}{2}} \sqrt{G'(x)^2+G''(x)^2} \, dx &= \int_{-\frac{\pi_{\mathrm{freq}}}{2}}^{\frac{\pi_{\mathrm{freq}}}{2}} \sqrt{G'(x)^2+G(x)^2} \, dx\\ 
&=\int_{-\frac{\pi_{\mathrm{freq}}}{2}}^{\frac{\pi_{\mathrm{freq}}}{2}} 1 \, dx\\ 
&= \pi_{\mathrm{freq}}.
\end{align*}
This completes the proof.
\end{proof}
 We note that $\pi_{\mathrm{amp}}$ is the well-known Wallis product for $\pi$. 
\section{Further Work} 

\begin{itemize} 

\item See if there is a bijection to establish the identity 
\[
\frac{d}{dx} \sin(x) = \cos(x)
\]
using their product formulas.

\item Find a bijection for the equality
\[
\pi_{\mathrm{freq}} = \pi_{\mathrm{amp}}. 
\]

\item Find a direct proof of equation \eqref{g g' 1} without using calculus if one is not known.

\item See if bijections can be found in evaluating the MZV-like series corresponding to the power series and their derivatives
\[
\sum_{n=0}^\infty (-1)^n\frac{x^{mn}}{(mn)!} = \frac{1}{m}\sum_{k=0}^{m-1} e^{x e^{\frac{\pi i }{m}} e^{\frac{2 \pi k i }{m}}}
\]
for even integer $m>1$. For odd integer $m>1$ see if bijections hold for $\zeta(\{ m\}^k)$ with either complete factorization or as a sum of terms. 

\item The evaluation of $\zeta(\{ 2\}^k)$ as a rational multiple of $\pi^{2k}$ implies that $\zeta(2k)$ is also a rational multiple of $\pi^{2k}$, for example, by the Newton-Girard identities. Use these identities or the generating function of the Bernoulli numbers to find expressions for $|B_{2k}|$ that show its positivity. Expressions showing positivity could be useful for studying the quantities 
\[
\sum_{n=1}^\infty \frac{e^{-\pi n^2}}{(\pi n^2)^k}
\]
for $k \geq 1$. We would like to express the above quantity using rational approximations that also imply positivity. One could use the sum 
\[
e^{-\pi n^2} \approx \sum_{m=0}^M  \frac{(-\pi n^2)^m}{m!}
\]
and then use rational approximations for $\pi$, but the initial terms in the sum oscillate largely from positive to negative and one has to take $M$ so large depending on $n$ until the sum becomes a good approximation. However we may use 
\begin{align*}
e^{-\pi n^2} &= e^{-\frac{1}{1-(1-\frac{1}{\pi n^2})}} \\ 
&= e^{-\sum_{m=0}^\infty (1-\frac{1}{\pi n^2})^m}\\ 
&= \prod_{m=0}^\infty e^{-(1-\frac{1}{\pi n^2})^m}\\
\end{align*}
And $e^{-v}$ for $0 <v <1$ can be written as a series with positive terms using the  derangement numbers. Then the series
\[
\sum_{n=1}^\infty \frac{1}{(\pi n^2)^k}(1-\frac{1}{\pi n^2})^m
\]
can be evaluated using $B_{2k}$ as a sum involving $\frac{\zeta(2k)}{\pi^k}$ which becomes a rational polynomial in $\pi$ instead of $\pi^2$. This move from $\pi^{2}$ to $\pi$ may suggest to go from using the series for $\pi^{2}$ as $6 \zeta(2)$ to the Wallis product. 

These rational expressions can be applied to the Taylor series coefficients $a_n$ of $\xi(s)$ using the formula for the upper incomplete Gamma function found in \cite{DeFranco 2}. Perhaps a these rational expressions can be used to obtain a third proof of the positivity of the $a_n$, apart from the one using the Polya-Schur kernel or the one found in \cite{DeFranco 1}.  

\item Apply this bijection to $q$-analogues of $\sin(x)$. The $q$-analogue of Gosper using the Wallis product would correspond to a $q$-analogue of $\pi_{\mathrm{amp}}$. 

 \end{itemize}

\end{document}